\documentclass[a4paper,11pt]{article}

\usepackage{amsthm,amsmath}
\usepackage[%
bookmarks=true,
ocgcolorlinks,%
linkcolor=blue,%
filecolor=blue,%
citecolor=blue,%
urlcolor=blue]{hyperref}
\usepackage{amssymb}
\usepackage{mathtools}
\usepackage{graphicx}
\usepackage{psfrag}
\usepackage{verbatim}

\newcommand{\bbE}{\mathbb{E}}

\newcommand{\bbP}{\mathbb{P}}

\newcommand{\bbR}{\mathbb{R}}

\newcommand{\ul}{\underline}

\newcommand{\cD}{\mathcal D}
\newcommand{\cA}{\mathcal A}

\newcommand{\cC}{\mathcal C}

\newcommand{\cB}{\mathcal B}
\newcommand{\cG}{\mathcal G}

\newcommand{\cS}{\mathcal S}

\newcommand{\cN}{\mathcal N}

\newcommand{\cQ}{\mathcal Q}

\newcommand{\rmd}{{\rm d}}
\newcommand{\rme}{{\rm e}}

\newcommand{\wt}{\widetilde}

\newtheorem{theorem}{Theorem}[section]

\newtheorem{lem}[theorem]{Lemma}
\newtheorem{prop}[theorem]{Proposition}

\newtheorem{rem}[theorem]{Remark}

\numberwithin{equation}{section}

\title{Decorated Random Walk Restricted to Stay Below a Curve (Supplement Material)}

\author
{Aser Cortines\thanks{aser.cortinespeixoto@math.uzh.ch, lhartung@uni-mainz.de,  oren.louidor@gmail.com.} \\Universit\"at Z\"urich \and Lisa Hartung\footnotemark[1]\\ 
Universit\"at Mainz\and Oren Louidor\footnotemark[1]\\Technion, Israel}
\date{}

\begin{document}
\maketitle

\begin{abstract}
We consider a one dimensional random-walk-like process, whose steps are centered Gaussians with variances which are determined according to the sequence of arrivals of a Poisson process on the line. This process is decorated by independent random variables which are added at each step, but do not get accumulated. We study the probability that such process, conditioned to form a bridge, stays below a curve which grows at most polynomially fast away from the boundaries, with exponent less than one half. Both bounds and asymptotics are derived. These estimates are used in~\cite{cortines2017structure}, to which this manuscript is a supplement.
\end{abstract}

\section{Introduction and statement of results}
\subsection{Introduction}
This supplement includes estimates for the probabilities of events involving a certain random-walk-type process restricted to stay below a ``barrier'' curve. These estimates are used as inputs to the proofs in~\cite{cortines2017structure}, to which this manuscript is a supplement, where the detailed structure of extreme level sets of branching Brownian motion is studied. The particular setup, which will be described shortly, is encountered naturally when one observes the process from the point of view of a particle, which reaches an extremal height (see Sub-section 1.3.1 in~\cite{cortines2017structure} for more information concerning the context of the results presented here). 

Although barrier estimates for Brownian motion and random walks are well known, our setup and the type of estimates we need -- both rather specific, are not addressed, as out-of-the-box statements, by the existing literature (at least to the extent of our knowledge). Thus, despite eventually recovering, up to constants, the same bounds and asymptotics as, say, a Brownian bridge restricted to stay negative, all estimates here are derived almost from scratch.

We only make a modest attempt to generalize the results to a setting larger than that which is needed in~\cite{cortines2017structure}. We feel that any reasonable generalization will still possess the particularities of the context at hand, and as such will be of little use anywhere else. Nevertheless, although we could have done with much less, we treat here all barrier curves which grow polynomially-fast away from the boundary, with exponent less than one half. This is almost an optimal assumption, as even in the case of Brownian motion, the probability of staying below a curve growing polynomially with exponent greater than one half, exhibits very different asymptotics compared to that of just staying negative. 

Our proofs essentially follow the strategy of reducing the problem at hand to that of a Brownian motion staying below such polynomial curves. Surprisingly enough, even for the case of Brownian motion, we could not find satisfactory upper and lower bounds, which cover the full range of exponents we want (c.f.~\cite{biskup2015full, Bramson1978maximal}, the exponent in the lower bound in~\cite{biskup2015full}, does go all the way to one half). For this reason and since, thanks to the advice of the referee, short proofs for such estimates are available following Bramson's argument in~\cite{B_C}, we have included these statements and their proofs in Subsection~\ref{ss:1}. We remark that Bramson's method was already put in use before (e.g., in~\cite{ABK_G} to obtain localization results for the path taken by extreme particles).

\subsection{Setup and results}
\label{ss:Setup}
The setting is as follows. Let $W=(W_u :\: u \geq 0)$ be a standard one dimensional Brownian motion. Given $x,y \in \bbR$  and $0 \leq s < t$, we shall denote by $\bbP_{s,x}^{t,y}$ and
$\bbP_{s,x}$ the conditional distribution $\bbP(\cdot \,|\, W_s = x \,,\,\, W_t = y)$
and $\bbP(\cdot \,|\, W_s = x)$ respectively (if $s=0$ we assume that $W_0$ was $x$ in the first place). On the same probability space, let us suppose also the existence of a collection $Y=(Y_u :\: u \geq 0)$ of independent random variables, which is also independent of $W$. These random variables, which will be referred to as ``decorations'', satisfy 
\begin{equation}
\label{e:A1}
\forall u,z \geq 0 :\ \ 
\bbP \big(|Y_u| \geq z) \leq \delta^{-1} \rme^{-\delta z} 
\end{equation}
for some $\delta > 0$.

The third collection of random variables defined on this space, comes in the form of a Poisson point process (PPP) on $\bbR$:
\begin{equation}
\label{e:A2}
\cN \sim \text{PPP}(\lambda \rmd x) \,,
\end{equation}
for some $\lambda > 0$. This process is assumed to be independent of $W$ and $Y$ and we denote by $\sigma = (\sigma_k :\: k \geq 1)$ the collections of all atoms of $\cN$, enumerated in increasing order. 

We will be interested in controlling the probability that the process $W-Y$ evaluated at all points $\sigma \cap [0,t]$ stays below a curve $\gamma_t = (\gamma_{t,u} :\: u \geq 0)$, satisfying for all $0 \leq u \leq t$,
\begin{equation}
\label{e:A3}
-\delta^{-1} \leq \gamma_{t,u} \leq \delta^{-1} \big(1+ (\wedge^t(u))^{1/2-\delta}\big) 
\quad, \qquad
\wedge^t(u) := u \wedge (t-u) \,,
\end{equation}
where $\delta \in (0,1/2)$ (to avoid using too many parameters we will use one $\delta$ in multiple conditions). The first statement is an upper bound. In this case, we might as well use the bounding function as the barrier curve itself.
\begin{prop}
\label{p:A2}
Suppose that $W, Y, \cN$ are defined as above with respect to some $\lambda > 0$ and $\delta \in (0,1/2)$. Then there exists $C = C(\lambda, \delta)$ such that for all $t \geq 0$, $x,y \in \bbR$,
\begin{equation}
\label{e:2.11}
\begin{multlined}
\bbP_{0,x}^{t,y} \Big( \max_{k :\:\sigma_k \in [0,t]} \big(W_{\sigma_k} - \delta^{-1} \big(1+(\wedge^t(\sigma_k))^{1/2-\delta}\big) - Y_{\sigma_k}\big) \leq 0 \Big) \\
\leq C \frac{(x^-+1)(y^-+1)}{t} \,,
\end{multlined}
\end{equation}

Moreover, there exists $C'= C'(\lambda, \delta)$ such that
for all $t \geq 0$ and all $x,y \in \bbR$ such that $xy \leq 0$,

\begin{equation}
\label{e:2.12}
\begin{multlined}
\bbP_{0,x}^{t,y} \Big( \max_{k:\: \sigma_k \in [0,t]} \big(W_{\sigma_k} - \delta^{-1} \big(1+(\wedge^t(\sigma_k))^{1/2-\delta}\big) - Y_{\sigma_k}\big) \leq 0 \Big) \\
\leq C' \frac{\big(x^- + \rme^{-\sqrt{2\lambda}(1-\delta) x^+}\big) 
\big(y^- + \rme^{-\sqrt{2\lambda}(1-\delta) y^+}\big)}{t} 
\exp\Big(\tfrac{(y-x)^2}{2t}\Big) \,.
\end{multlined}
\end{equation}
\end{prop}

\medskip
For an asymptotic statement, we naturally need to control the limiting behavior of both the decorations and the family of curves $\gamma = (\gamma_t)_{t \geq 0}$. For the former we assume that
\begin{equation}
\label{e:A2.5}
Y_u \overset{u \to \infty} \Longrightarrow Y_\infty \,,
\end{equation}
in the sense of weak convergence of distributions, where $Y_\infty$ is some real random variable. For the latter, we require that for all $u \geq 0$,
\begin{equation}
\label{e:A4}
\gamma_{t,u} \overset{t \to \infty} \longrightarrow \gamma_{\infty, u}
\ , \ \ 
\gamma_{t,t-u} \overset{t \to \infty}  \longrightarrow \gamma_{\infty, -u} \,,
\end{equation}
where $\gamma_{\infty, u}, \gamma_{\infty, -u} \in \bbR_+$ (with slight abuse, we shall use the notation $\gamma_{\infty, -0}$ for the limit $\lim_{t \to \infty} \gamma_{t,t}$). 

We then have (henceforth ``$\sim$'' denotes asymptotic equivalence):
\begin{prop}
\label{p:A3}
Suppose that $W, Y, \cN$ and $\gamma$ are defined as above with respect to some $\lambda>0$ and $\delta \in (0,1/2)$. Then there exists non-increasing positive functions $f,g: \bbR \to (0,\infty)$ depending on $\lambda$, $\gamma$ and $Y$, such that 
\begin{equation}
\label{e:A6}
\bbP_{0,x}^{t,y} \Big( \max_{k :\: \sigma_k \in [0,t]} \big(W_{\sigma_k} - \gamma_{t,\sigma_k} - Y_{\sigma_k} \big) \leq 0 \Big)
\sim 2 \frac{f(x)g(y)}{t} 
\ \ \text{as } t \to \infty \,,
\end{equation}
uniformly in $x,y$ satisfying $x, y \leq 1/\epsilon$ and $(x^- + 1)(y^- +1) \leq t^{1-\epsilon}$, for any fixed $\epsilon > 0$. Moreover, 
\begin{equation}
\label{e:A6.1}
\lim_{x \to \infty} \frac{f(-x)}{x} = \lim_{y \to \infty} \frac{g(-y)}{y} = 1 \,.
\end{equation}
\end{prop}

\begin{rem}[Monotonicity w.r.t. boundary conditions] \label{r:monotonicity} 
Notice that if $x \leq x'$ and $y \leq y'$, then for all $t \geq 0$ we have
\begin{equation}\label{e:montonicity}
\begin{multlined}
\bbP_{0,x}^{t,y} \Big( \max_{k :\: \sigma_k \in [0,t]} \big(W_{\sigma_k} - \gamma_{t,\sigma_k} - Y_{\sigma_k} \big) \leq 0 \Big) \\
\geq \bbP_{0,x'}^{t,y'} \Big( \max_{k :\: \sigma_k \in [0,t]} \big(W_{\sigma_k} - \gamma_{t,\sigma_k} - Y_{\sigma_k} \big) \leq 0 \Big).
\end{multlined}
\end{equation} 
Indeed, one can pass from a Brownian bridge from $x$ to $y$ to a Brownian bridge from $x'$ to $y'$ replacing $W_s$ by
$W_s - \Big( \frac{s}{t} (y'-y) +(x'-x) \big(1-\frac{s}{t} \big) \Big)$
inside the probability brackets. Since the above interpolation function is positive for every $s \in [0,t]$ we can simply lower bound it by zero to obtain \eqref{e:montonicity}. In particular, it is straightforward to show that if the convergence from Proposition~\ref{p:A3} holds, then both $f$ and $g$ are non-increasing.   
\end{rem}

We also need to know that the above asymptotics is continuous (in the sense specified below) in
$Y$ and $\gamma$. To this end for each $r \geq 0$, let $Y^{(r)}$ be a collection of random variables as $Y$ above and $\gamma^{(r)}$ be a function as $\gamma$ above, satisfying~\eqref{e:A1} and~\eqref{e:A3} uniformly for all $r \geq 0$ with some $\delta \in (0,1/2)$. Suppose that~\eqref{e:A2.5} holds for $Y^{(r)}_u$ with the limit denoted by $Y_\infty^{(r)}$ and that~\eqref{e:A4} holds with the limits denoted by $\gamma^{(r)}_{\infty, u}$ and $\gamma^{(r)}_{\infty,- u}$. 
Then
\begin{prop}
\label{p:A4}
Suppose that $W, Y, \cN, \gamma$ and $Y^{(r)}, \gamma^{(r)}$ for $r \geq 0$ are defined as above with respect to some $\lambda>0$, $\delta \in (0,1/2)$. Let $f^{(r)}$, $g^{(r)}$ be the functions $f$, $g$ respectively given in Proposition~\ref{p:A3} applied to $W, Y^{(r)}, \cN$ and $\gamma^{(r)}$. Assume that
\begin{equation}
\label{e:543}
\forall u \in [0,\infty]\!:\, Y^{(r)}_u \overset{r \to \infty}{\Longrightarrow} Y_u
\quad , \qquad
\forall u \in [0,\infty)\!:\,
\gamma^{(r)}_{\infty,\pm u} \overset{r \to \infty}{\longrightarrow} \gamma_{\infty, \pm u} \,.
\end{equation}
Then for all $x \in \bbR$,
\begin{equation}
f^{(r)}(x) \overset{r \to \infty} \longrightarrow f(x) 
\ , \quad
g^{(r)}(x) \overset{r \to \infty} \longrightarrow g(x) \,,
\end{equation}
with $f,g$ given by Proposition~\ref{p:A3} applied to $W, Y, \cN$ and $\gamma$. Moreover, if $Y_\infty^{(r)} = Y_\infty$ and $\gamma_{\infty, -u}^{(r)} = \gamma_{\infty, -u}$ for all $r \geq 0$ and $u \geq 0$, then $g^{(r)}(x) = g(x)$ for all $r \geq 0$.	
\end{prop}

\medskip
Finally, we show that conditioned to stay below the curve, the walk $W-Y$ gets repelled from it. This sharp entropic-repulsion-type result, which is not needed in~\cite{cortines2017structure}, will be employed in the sequel paper~\cite{cortines2017structure2} and hence we might as well include it here. The sharpness of the statement necessitates some additional regularity conditions for $\gamma$. For our purposes it will suffice to assume that (see Figure~\ref{f:Curve}):
\begin{equation}
\begin{split}
\label{e:A3s}
\gamma_{t,u} - \frac{u}{r} \gamma_{t,r} & \leq \delta^{-1} \big(1+(\wedge^r(u))^{1/2-\delta}\big)
\\
\gamma_{t,u'} - \frac{t-u'}{t-r} \gamma_{t,r} & \leq \delta^{-1} \big(1+(\wedge^{t-r}(u'-r))^{1/2-\delta}\big)\,,
\end{split}
\end{equation}
for all $0 < u < r < u' < t$. As usual, we also assume~\eqref{e:A1},~\eqref{e:A2} and~\eqref{e:A3}. Since we are not after asymptotics, we need not assume anything else. Under these assumptions we then have,
\begin{prop}
\label{p:A5}
Suppose that $W, Y, \cN$ and $\gamma$ are defined as above with respect to some $\delta \in (0,1/2)$ and $\lambda>0$. Then for all $M \geq 1$ and $x,y \in \bbR$, 
\begin{equation}
\nonumber
\varlimsup_{t \to \infty} \sqrt{s} \, \bbP_{0,x}^{t,y} \Big( \max_{u \in [s, t-s]} W_u - \gamma_{t,u} \geq -M \Big|
\max_{\sigma_k \in [0,t]}\big(
W_{\sigma_k} - \gamma_{t,\sigma_k} - Y_{\sigma_k}\big) \leq 0 \Big) 
\end{equation}
is bounded above uniformly in $s \geq 1$.
\end{prop}

\begin{figure}[ht!]
\centering
    \includegraphics[width=0.95\textwidth]{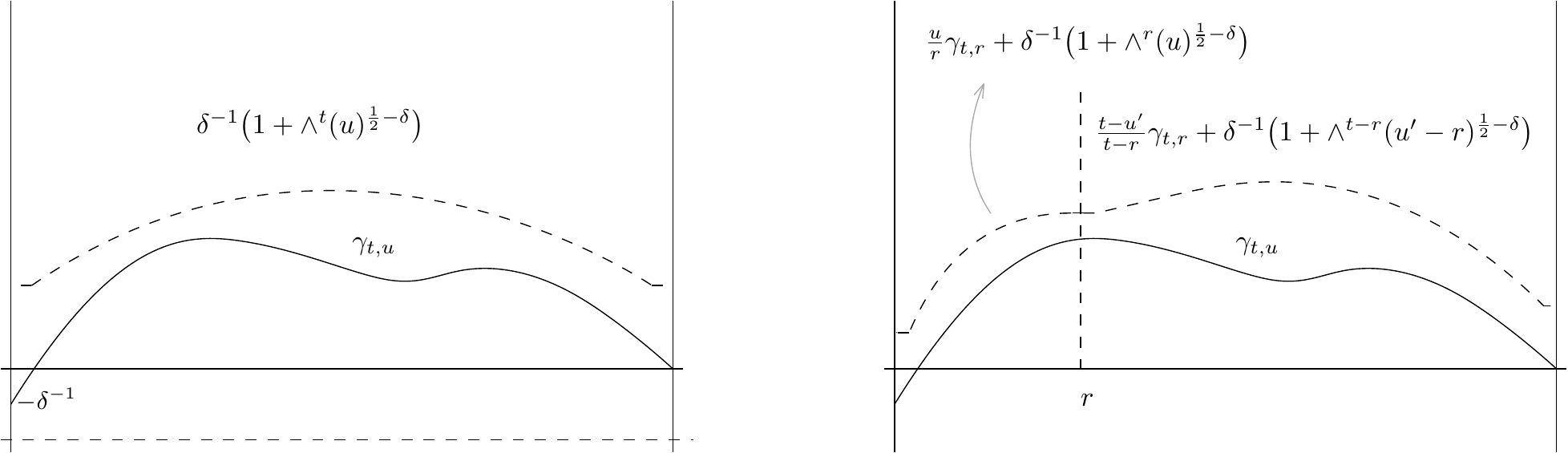} 
    \caption{Illustrations of boundedness condition~\eqref{e:A3} (left) and regularity condition~\eqref{e:A3s} (right), imposed on $u \mapsto \gamma_{t,u}$ as assumptions in Proposition~\ref{p:A5}. Only~\eqref{e:A3} is assumed in Proposition~\ref{p:A4}.
} \label{f:Curve}
\end{figure}

\section{Proofs}
\subsection{Brownian motion below a curve}
\label{ss:1}
The starting point for all proofs is the following upper and lower bounds for the probability that a standard Brownian motion stays below a barrier curve, whose absolute value grows polynomially with exponent smaller than $1/2$. Since we are only after upper and lower bounds, without loss of generality, we can use the largest, resp. smallest possible curves in our statements. We treat both Brownian bridge and Brownian motion and naturally recover the same bounds, up to constants, as in the  case of the barrier curve being the zero function.

Our motion will be assumed to start and terminate strictly below the curve with the constants in the bounds depending on how close the motion and the curve are allowed to be at the boundary. We thus only care about staying below the curve ``away from the boundary'' (staying below the curve near the boundary, which is not needed for our purposes, requires an additional analysis).

 Although the estimates we present here are rather folklore by now (c.f.~\cite{Bramson1978maximal} or, in more generality,~\cite{biskup2015full}), we could not find a reference which covers the full range of exponents we want and hence proofs are included as well. We follow Bramson's arguments in~\cite{B_C} which rely on the Cameron-Martin transformation. 
This makes all arguments rather simple. Lastly, to avoid working with negative values, we will actually consider the symmetric case of a Brownian motion staying above a curve.

\begin{prop}
\label{p:3.6}
For any $\delta \in (0,1/2)$, there exists $C, C' > 0$ such that for all $x,y > \delta$ and $t > 0$,
\begin{equation}
\label{e:3.18}
\bbP_{0,x}^{t,y} \Big(\min_{u \in [0,t]} \big(W_u + \delta^{-1} (\wedge^t(u))^{1/2-\delta}\big)  \geq 0 \Big)
\leq C \frac{xy}{t} \,,
\end{equation}
and for all $x,y > \delta$ and $t > 0$ such that $xy < \delta^{-1} t$,
\begin{equation}
\label{e:3.19}
\bbP_{0,x}^{t,y} \Big(\min_{u \in [0,t]} \big(W_u - \delta^{-1} (\wedge^t(u))^{1/2-\delta} \big) \geq 0 \Big)
\geq C' \frac{xy}{t} \,.
\end{equation}
\end{prop}

A one sided version of the above proposition is
\begin{prop}
\label{p:3.7}
For any $\delta \in (0,1/2)$, there exists $C, C' > 0$ such that for all $x > \delta$ and $t > 0$,
\begin{equation}
\label{e:3.18a}
\bbP_{0,x} \Big(\min_{u \in [0,t]} \big(W_u + \delta^{-1} u^{1/2-\delta}\big)  \geq 0 \Big)
\leq C \frac{x}{\sqrt{t}}  \,.
\end{equation}
and for all $x > \delta$ and $t > 0$ such that $x < \delta^{-1} \sqrt{t}$,
\begin{equation}
\label{e:3.19a}
\bbP_{0,x} \Big(\min_{u \in [0,t]} \big(W_u - \delta^{-1} u^{1/2-\delta} \big) \geq 0 \Big)
\geq C' \frac{x}{\sqrt{t}} \,.
\end{equation}
\end{prop}

\begin{proof}[Proof of Proposition~\ref{p:3.6}]
First observe that we can always assume that $t > \delta^3$. In the lower bound case, this follows from the restriction on $x$,$y$ and $t$. For the upper bound, we can always adjust $C$ so that if $t \leq \delta^3$ the right hand side of~\eqref{e:3.18} is larger than $1$. Now, denote by $\cC_{0,0}^{t,0}$ the set of all continuous functions from $[0,t]$ to $\bbR$ vanishing at $0$ and $t$. Set $\eta := 1/2-\delta$ and choose $\eta' > 1/2$ such that $\eta + \eta' < 1$. Define for all $x,y \geq 0$, $t > \delta^3$ and $M > 0$ the subsets:
\begin{equation}
\begin{split}
\cB_{0,x}^{t,y}  & := \Big\{ w \in \cC_{0,0}^{t,0} :\: \min_{u \in [0,t]} \big(w_u + (1-u/t)x + (u/t) y\big) \geq 0 \Big\} \,, \\
\cA^t_M & := \Big\{ w \in \cC_{0,0}^{t,0} :\: \max_{u \in [\delta^{3/\eta},t-\delta^{3/\eta}]} \big(w_u - M (\wedge^t(u))^{\eta'} \big) \leq 0 \Big\} \,.
\end{split}
\end{equation}

For the upper bound, we need to consider $\bbP_{0,0}^{t,0} \big( (W + \delta^{-1} (\wedge^t)^\eta) \in \cB_{0,x}^{t,y} \big)$. As a first step, we claim that if $M$ is chosen large enough, we can bound from above instead the probability
\begin{equation}
\label{e:3.21}
\bbP_{0,0}^{t,0} \Big( W + \delta^{-1} (\wedge^t)^\eta \in \cB_{0,x}^{t,y} \cap \cA^t_M \Big) \,.
\end{equation}
Indeed, by the strong FKG property of $W$ (namely, monotonicity of the conditional law of $W$ with respect to the curve above which it is conditioned to stay; see Lemma 2.6 in~\cite{B_C}), we have
\begin{equation}
\begin{split}
\label{e:2.7}
\bbP_{0,0}^{t,0}  \Big( W + \delta^{-1} (\wedge^t)^\eta \in \cA^t_M \,\Big|& \, 
 W + \delta^{-1}  (\wedge^t)^\eta \in \cB_{0,x}^{t,y} \Big) \\
&\geq \bbP_{0,0}^{t,0} \big( W + \delta^{-1} (\wedge^t)^\eta \in \cA^t_M \,\Big|\, 
 W \in \cB_{0,\delta}^{t,\delta} \Big) \\
&\geq \bbP_{0,0}^{t,0} \Big( W \in \cA^t_{M-\delta^{-1}} \,\Big|\, 
 W \in \cB_{0,\delta}^{t,\delta} \Big) \,.
\end{split}
\end{equation}
Since $\eta' > 1/2$, the last probability is bounded away from $0$ uniformly in $t$, if $M$ is chosen sufficiently large. This follows, e.g. from the proof of Lemma 2.7 in~\cite{B_C} (just sum the probabilities in the display above (2.24) and take $C \to \infty$).

Next, for $0 \leq u \leq t$ and $t > \delta^3$, let $\beta_{t,u}$ be defined via
\begin{equation}
\beta_{t,u} := 
\begin{cases}
	0 & u \in \big[0 ,\, \delta^{3/\eta}\big] \,,\\
	\delta^{-1} \eta u^{\eta - 1}		& u \in \big(\delta^{3/\eta}, \,t/2\big] \,,\\
	-\delta^{-1} \eta(t-u)^{\eta - 1} 	& u \in \big(t/2,\, t-\delta^{3/\eta}\big] \,, \\
	0 						& u \in \big(t-\delta^{3/\eta},\, t\big] \,,
\end{cases}
\end{equation}
and set $\alpha_{t,s} := \int_0^s \beta_{t,u} \rmd u$. Observing that $\delta^{-1} (\wedge^t(s))^\eta - \alpha_{t,s} \in [0,\delta/2]$ for all $0 \leq s \leq t$ and denoting by $\alpha_t$ the function $s \mapsto \alpha_{t,s}$, we can upper bound~\eqref{e:3.21} by $\bbP_{0,0}^{t,0} \Big( W + \alpha_t \in \cB_{0,x_+}^{t,y_+} \cap \cA^t_M \Big)$, where $x_+ := x+\delta/2$ and $y_+ := y+\delta/2$.

Now using the Cameron-Martin formula (see, e.g. Section 1.4 in~\cite{morters2010brownian}) for the Radon-Nykodim derivative $\rmd \bbP_{0,0}^{t,0} (W + \alpha_t \in \cdot) / \rmd \bbP_{0,0}^{t,0} (W \in \cdot)$, the last probability is equal to 
\begin{equation}
\exp \Big[-\frac12 \int_0^t \beta^2_{t,u} \rmd u + \frac{1}{2t} \Big(\int_0^t \beta_{t,u} \rmd u\Big)^2\Big] \bbE_{0,0}^{t,0} \Big[\rme^{\int_0^t \beta_{t,u} \rmd W_u} ;\; 
W \in \cB_{0,x_+}^{t,y_+} \cap \cA^t_M \Big] \,.
\end{equation}
The exponent outside the expectation is bounded by a constant uniformly in $t$. Since $\beta_{t,u}$ has bounded variations in $u$, we can use the integration by parts formula to write the last mean as 
\begin{equation}
\bbE_{0,0}^{t,0} \Big[ \exp\Big(-\int_0^t W_u \rmd \beta_{t,u} \Big) ;\; W \in \cB_{0,x_+}^{t,y_+} \cap  \cA^t_M \Big] \,.
\end{equation}
The relation $\eta+\eta' < 1$ guarantees that for any $M$, whenever $W \in \cA^t_M$ the integral is bounded by a uniform constant, regardless of $t$. This implies that the last expectation is at most $C \bbP_{0,x_+}^{t,y_+} \big(\min_{u \in [0,t]} W_u \geq 0\big)$, which is bounded above by $C (x_+y_+)/t \leq C'xy/t$ thanks to the Ballot Theorem for Brownian motion. This shows~\eqref{e:3.18}.

The argument for~\eqref{e:3.19} is quite similar. Since we are after a lower bound, we can freely 
require that $W-\delta^{-1} (\wedge^t)^\eta \in \cA^t_M$ for any $M > 0$. Using the definition of $\beta_{t,u}$ and $\alpha_{t,s}$ from before, it suffices to bound from below the probability
\begin{equation}
\bbP_{0,0}^{t,0} \Big( W - \alpha_t \in \cB_{0,x_-}^{t,y_-} \cap \cA^t_M \Big) \,,
\end{equation}
where $x_- := x-\delta/2$ and $y_- := y-\delta/2$. Using Cameron-Martin again, this time for the law of $W + \alpha_t$ and hence with $-\beta_{t,u}$ in place of $\beta_{t,u}$, followed by integration by parts as before, the latter is at least a constant times
\begin{equation}
\nonumber
\bbE_{0,0}^{t,0} \Big[ \exp\Big(\int_0^t W_u \rmd \beta_{t,u} \Big) ;\; W \in \cB_{0,x_-}^{t,y_-} \cap  \cA^t_M \Big] 
\geq C \bbP_{0,0}^{t,0} \big(W \in \cB_{0,x_-}^{t,y_-} \cap  \cA^t_M \big) \,,
\end{equation}
where $C$ depends on $M$.

Lastly, we use the strong FKG property and argue as in~\eqref{e:2.7} to claim,
\begin{equation}
\bbP_{0,0}^{t,0} \big(W \in \cA^t_M \,\big|\, W \in \cB_{0,x_-}^{t,y_-} \big)
\geq \bbP_{0,0}^{t,0} \big(W \in \cA^t_M \,\big|\, W \in \cB_{0,\delta/2}^{t,\delta/2} \big) \geq C \,,
\end{equation}
for some $C > 0$ which does not depend on $t$, $x$ or $y$, once $M$ is chosen large enough. Combining the last two displays, it remains to bound from below $\bbP_{0,x_-}^{t,y_-}(\min_{u \in [0,t]} W_u \geq 0)$. Thanks to again to the Ballot Theorem, this is at least $C (x_- y_-) / t \geq C'xy/(4t)$ whenever $xy/t < \delta^{-1}$ for some $C' > 0$ which depends on $\delta$.
\end{proof}

\begin{proof}[Proof of Proposition~\ref{p:3.7}]
The proof is similar to the proof of Proposition~\ref{p:3.6} and hence we only highlight the differences. Again we can assume without loss that $t > \delta^4$ and set $\eta := 1/2-\delta$. Instead of event $\cB_{0,x}^{t,y}$ and $\cA^t_M$ we now use  
\begin{equation}
\begin{split}
\wt{\cB}_{0,x}  & := \Big\{ w \in \cC_{0,0}^{t,0} :\: \min_{u \in [0,t]} \big(w_u + x\big) \geq 0 \Big\} \,, \\
\wt{\cA}^t_M & := \Big\{ w \in \cC_{0,0}^{t,0} :\: \max_{u \in [\delta^{4/\eta},t]} \big(w_u - M u^{\eta'} \big) \leq 0 \Big\} \,,
\end{split}
\end{equation}
for $t > \delta^4$, $M > 0$ and $\eta'$ chosen as before so that $\eta + \eta' < 1$. The function $(t,u) \mapsto \beta_{t,u}$ is now replaced by
\begin{equation}
\wt{\beta}_{t,u} := 
\begin{cases}
	0 & u \in [0 ,\, \delta^{4/\eta}\big] \,,\\
	\delta^{-1} \eta u^{\eta - 1}		& u \in \big(\delta^{4/\eta}, \,t\big] \,,
\end{cases}
\end{equation}
for $0 \leq u \leq t$ and $t > \delta^4$.

Proceeding as before, we replace the barrier curve by the indefinite integral of $u \mapsto \wt{\beta}_{t,u}$ and freely add the restriction that paths lie in $\wt{\cA}^t_M$ (this requires one sided analogs of Lemmas 2.6 and 2.7 from~\cite{B_C}, which can be proved using similar arguments). Using the Cameron-Martin formula for the unconditional Brownian motion, we are then led to bounding
\begin{equation}
\label{e:3.24}
\exp \Big[-\frac12 \int_0^t \wt{\beta}^2_{t,u} \rmd u \Big] \bbE_{0,0} \Big[\rme^{\pm \int_0^t \wt{\beta}_{t,u} \rmd W_u} ;\; 
W \in \wt{\cB}_{0,x_\pm} \cap \wt{\cA}^t_M \Big] \,,
\end{equation}
with $x_\pm$ defined as before. The integral in the first exponent is bounded uniformly in $t$, exactly as before. For the stochastic integral, we use integration by parts and thus equate it to
\begin{equation}
\nonumber
\pm \wt{\beta}_{t,t} W_t \mp \int_0^t W_u \rmd \wt{\beta}_{t,u}  \,,
\end{equation}
which is bounded uniformly in $t$ on $\wt{\cA}^t_M$, for any choice of $M$ since $\eta + \eta' < 1$.

Thus as before, the probabilities in~\eqref{e:3.18a} and~\eqref{e:3.19a} are bounded above and below respectively by a constant times $\bbP_{0,x_\pm}(\min_{u \in [0,t]} W_u \geq 0)$. It remains to use the Reflection Principle for Brownian motion, to assert that the latter probability is bounded above and below by $Cx/\sqrt{t}$ under the restrictions on $x$. 
\end{proof}

\subsection{The Upper Bound (Proof of Proposition~\ref{p:A2})}
The proof follows from the sequence of lemmas given below. For the first of which we let $t, \delta > 0$, $M > 1$ and suppose that $\ul{s} = (s_k)_{k=0}^n$ for some $n \geq 1$ is a sequence of real numbers, such that with $\Delta_k : = s_k - s_{k-1}$ it satisfies:
\begin{description}
\item[(A1)] $0 = s_0 < s_1 < \dots < s_n = t$,
\item[(A2)] $\Delta_k \leq \delta^{-1} \log^+ k + M$, for all $k=1, \dots, n$.
\item[(A3)] $\delta k - M \leq s_k \leq \delta^{-1} k + M$, for all $k=0, \dots, n$.
\end{description}
We denote by $\cS(\delta, t, M)$ the collection of all such sequences with arbitrary length $n$. 
Given such sequence $\ul{s}=(s_k)_{k=0}^n$, we define $\wt{\ul{s}} = (\wt{s}_k)^n$, by setting $\wt{s}_k := t - s_k$ and also $\wt{\Delta}_k := \wt{s}_k - \wt{s}_{k-1}$. Finally, we let $\wt{\cS}(\delta, t, M)$ be the collection of all sequence $\ul{s}$ such that both $\ul{s}$ and $\wt{\ul{s}}$ are in $\cS(\delta,t, M)$.
\begin{lem}
Let $\delta \in (0,1/2), t > 0$, $M > 1$ and suppose that $\ul{s} = (s_k)_{k=0}^n$ is a sequence in $\wt{\cS}(\delta, t, M)$. Then there exists $C = C(\delta)$ such that for all $x,y \leq 0$
\begin{equation}
\label{e:A5}
\bbP_{0,x}^{t,y} \Big( \max_{k \in [1,n-1]} \big(W_{s_k} - \delta^{-1} (\wedge^t(s_k))^{1/2-\delta} \big) \leq 0 \Big) \leq C M^2 \frac{(x^-+1)(y^-+1)}{t} \,.
\end{equation}
\end{lem}
\begin{proof}
Noting that all constants in this proof depend only on $\delta$, we set $\delta^{'-1} := 3\delta^{-2(1-\delta)}$ and $M' := C_0(M+1)$ where $C_0 >0$ is a constant to be determined later. We also let $u_k : = s_k \wedge (t-s_{k-1})$. Then, using stochastic monotonicity of the trajectories of $W$ under $\bbP_{0,x}^{t,y}$ in $x,y$, the left hand side in~\eqref{e:A5} is bounded above by
\begin{equation}
\nonumber
 \frac{ \bbP_{0,x}^{t,y} \big(\sup_{[0,t]} \big(W_u - \delta^{'-1} (\wedge^t(u))^{1/2-\delta} \big) \leq M' \big)}
{1 - \sum\limits_{k=1}^n 
 \bbP \Big(\sup\limits_{[s_{k-1}, s_k]} W_u >  M' \!+   \tfrac{1}{\delta'} \! u_k^{1/2-\delta}\big) 
 \, \Big| \, W_{s_{k-1}}  \! =  \!  W_{s_k} \! = \tfrac{1}{\delta}  \! u_k^{1/2-\delta}  \Big)} \,.
\end{equation}
The sum in the denominator can be further dominated by
\begin{equation}
\label{e:A6.1.1}
\begin{multlined}
\sum_{k=1}^{n-1} \bbP_{0,0}^{\Delta_k, 0 } \Big(\sup_{[0, \Delta_k]} W_u >  
	\delta^{'-1} s_{k-1}^{1/2-\delta} - \delta^{-1} s_k^{1/2-\delta} + M'\Big) \\
+ \sum_{k=1}^{n-1} \bbP_{0,0}^{\wt{\Delta}_k, 0} \Big(\sup_{[0, \wt{\Delta}_k]} W_u >  
	\delta^{'-1} \wt{s}_{k-1}^{1/2-\delta} - \delta^{-1} \wt{s}_k^{1/2-\delta} + M'\Big)\,.
\end{multlined}
\end{equation}

By conditioning on the first time $W$ reaches a point $z$ and using the reflection principle, one has $\bbP_{0,0}^{s,0} (\sup_{[0,s]} W_u \geq z) 
= \rme^{-2z^2/s}$ for all $s > 0$ and $z \geq 0$.
Moreover, using assumption (A3), it is not difficult to see that for any $C_1 > 0$ we may choose $C_0$ large enough in the definition of $M'$, so that $\delta^{'-1} s_{k-1}^{1/2-\delta} - \delta^{-1} \! s_k^{1/2-\delta} + M'$ is at least
\begin{equation}
\nonumber
\delta^{'-1} \big(\delta (k-1) -M \big )^{1/2-\delta} - \delta^{-1} \big(\delta^{-1} k + M \big)^{1/2-\delta} + M'
\geq 2(\delta^{-1} \log^+ k + M) + C_1 \,.
\end{equation}
Using also (A2), the first sum in~\eqref{e:A6.1.1} is therefore bounded above by
\begin{equation}
C \sum_{k=1}^n k^{-2} \rme^{-C_1} < \frac{1}{4}\,,
\end{equation}
since $\delta < 1$ and provided we make sure that $C_1$ is large enough.

Since $\wt{\ul{s}}$ satisfies (A1)-(A3) as well, similar reasoning shows that the second sum in~\eqref{e:A6.1.1} can also be bounded by $1/4$. But then the first display in the proof together with Proposition~\ref{p:3.6} shows that the probability in question is at most
\begin{equation}
C \frac{(x^-+M')(y^-+M')}{t} \leq
C M^2 \frac{(x^-+1)(y^-+1)}{t} \,,
\end{equation}
as desired.
\end{proof}

Next we add the collection $Y$.
\begin{lem}
Let $\delta > 0$, $M > 1$ and $t > 0$. Suppose that $\ul{s} = (s_k)_{k=0}^n$ is a sequence in $\wt{\cS}(\delta, t, M)$, that $W, Y$ are defined as in Subsection~\ref{ss:Setup} and in particular that~\eqref{e:A1} holds. Then there exists $C = C(\delta)$ such that for all $x,y \leq 0$
\begin{equation}
\label{e:A7}
\bbP_{0,x}^{t,y} \Big( \max_{k \in [1,n-1]} \big(W_{s_k} - \delta^{-1} (\wedge^t(s_k))^{\frac12-\delta} - Y_{s_k} \big) \leq 0 \Big) \leq C M^3 \frac{(x^-+1)(y^-+1)}{t} \,.
\end{equation}
\end{lem}
\begin{proof}
Noting again that all constants in this proof will depend only on $\delta$, we bound the left hand side in~\eqref{e:A7} by
\begin{equation}
\label{e:A8}
\begin{multlined}
\sum_{m=1}^\infty \bbP_{0,x}^{t,y} \Big(\max_{k \in [1,n-1]}
	\big(W_{s_k} - 3\delta^{-1} (\wedge^t(s_k))^{1/2-\delta}\big) \leq m \Big) \\
	\times\bbP \Big( \max_{k \in [1,n-1]} \big(Y_{s_k} - 2\delta^{-1} \log^+ \wedge^t(s_k) \big) \in [m-1, m) \Big) \,,
\end{multlined}
\end{equation}
where we used that $(\wedge^t(s_k))^{1/2-\delta}>\log^+ \wedge^t(s_k)$. By the previous lemma, the first probability in \eqref{e:A8} is at most 
\begin{equation}\label{li.E2}
C t^{-1} M^2 (x^-+m)(y^-+m) \leq C t^{-1} M^2 m^2 (x^-+1)(y^-+1).
\end{equation} 
As for the second, using the bound on the tail of the variables in $Y$ and assumption (A3) for $\ul{s}, \wt{\ul{s}}$, we may bound it from above by
\begin{equation}
\begin{split}
\sum_{k=1}^{n-1} &\big( \bbP (Y_{s_k} \geq 2 \delta^{-1} \log^+ s_k + m-1)
+ \bbP ( Y_{\wt{s}_k} \geq 2 \delta^{-1} \log^+ \wt{s}_k + m-1) \big)\\
& \leq \delta^{-1} \sum_{k=1}^{n-1} \big( \rme^{-2 \log^+ s_k - \delta(m-1)} + 
	\rme^{-2 \log^+ s_k - \delta(m-1)} \big) \\
&\leq C \rme^{-\delta m} \sum_{k=1}^\infty \big((\delta k - M) \wedge 1\big)^{-2} 
\leq C M \rme^{- \delta m} \,.
\end{split}
\end{equation}
Using both bounds in~\eqref{e:A8} we recover~\eqref{e:A7}.
\end{proof}

The following lemma is essentially the first part of Proposition~\ref{p:A2}. 
\begin{lem}
\label{l:A1}
Let $\delta, \lambda >0$ be given and suppose that $W, Y, \cN$ are defined as in Subsection~\ref{ss:Setup} and in particular that~\eqref{e:A1} and~\eqref{e:A2} hold. Then there exists $C = C(\lambda, \delta)$ such that for all $t \geq 0$, $x,y \in \bbR$,
\begin{equation}
\label{e:197}
\bbP_{0,x}^{t,y} \Big( \max_{k:\: \sigma_k \in [0,t]} \big(W_{\sigma_k} - \delta^{-1} ( \wedge^t(\sigma_k))^{1/2-\delta} - Y_{\sigma_k}\big) \leq 0 \Big)
\leq C \frac{(x^-+1)(y^-+1)}{t} \,,
\end{equation}
\end{lem} 
\begin{proof}
Stochastic monotonicity of the trajectories of $W$ under $\bbP_{0,x}^{t,y}$ in $x,y$ implies that it is enough to prove the lemma for $x,y \leq 0$. Given a realization $\cN$ with $n-1 = \cN([0,t])$, we set $s_0 = 0$, $s_k = \sigma_k$ for $k=1, \dots, n-1$ and $s_n = t$. Set also $\delta' = \min \big\{\delta, 2\lambda/3, 1/(2\lambda)\big\}$ and finally let $\Delta \wt{S}(\delta, t, m) := \wt{\cS}(\delta', t, m+1) \setminus \wt{\cS}(\delta', t, m)$. Then the right hand side of~\eqref{e:197} is bounded from above by
\begin{equation}
\label{e:A9}
\begin{split}
\sum_{m=1}^\infty &
\bbP \big(\ul{s} \in \Delta \wt{S}(\delta, t, m)  \big) \\
& \times\bbP_{0,x}^{t,y} \Big( \sup_{k \in [1,n-1]} W_{s_k} - \delta^{'-1} (\wedge^t(s_k))^{1/2-\delta} - Y_{s_k} \leq 0 \,\Big|\, \ul{s} \in \Delta \wt{S}(\delta, t, m)\Big)
\end{split}
\end{equation}

By the previous lemma, we know that the second probability is at most $C(m+1)^3(x^-+1)(y^-+1)/t$ for some $C > 0$ which depends only on $\delta$. For the first probability in \eqref{e:A9}, we first observe that since $\ul{s}$ and $\wt{\ul{s}}$ have the same law, it can be bounded above by $2 \bbP \big(\ul{s} \notin \cS(\delta', t, m)\big)$. The last probability is further bounded above by
\begin{equation}
\sum_{k=1}^\infty \left( \bbP\left(\sigma_k - \sigma_{k-1} > \delta^{'-1} \log k + m\right)
+ \bbP \left(\left|\sigma_k - \frac{k}{\lambda}\right| > \frac{k}{2\lambda} + m \right)\right) \,.
\end{equation}

Since $\sigma_k - \sigma_{k-1}$ is exponentially distributed with rate $\lambda$, the first term in the sum is bounded above by $\rme^{-\lambda m} k^{-3/2}$. Using the exponential Chebychev inequality, it is not difficult to see that we may find constants $C, C' > 0$ such that the second term is at most $C \rme^{-C'(k + m)}$. The last two assertions imply that the sum and hence the first probability in~\eqref{e:A9} is exponentially decaying in $m$. Together with the bound on the second probability this shows what we wanted to prove.
\end{proof}

As a last step for the derivation of the upper bound we allow positive values for $x$.
\begin{lem}
\label{l:A2}
Let $\delta \in (0,1/2)$ and $\lambda >0$ be given and suppose that $W, Y, \cN$ are defined as in the Subsection~\ref{ss:Setup} and that in particular~\eqref{e:A1} and~\eqref{e:A2} hold. 
Then there exists $C = C(\lambda, \delta)$ such that
for all $t \geq 0$ and all $x,y \in \bbR$ with $xy \leq 0$,
\begin{equation}\label{e:A10}
\begin{multlined}
\bbP_{0,x}^{t,y} \Big( \max_{k:\: \sigma_k \in [0,t]} \big(W_{\sigma_k} - \delta^{-1}  (\wedge^t(\sigma_k))^{1/2-\delta} - Y_{\sigma_k}\big) \leq 0 \Big) \\
\leq C \frac{\big(x^- + \rme^{-\sqrt{2\lambda}(1-\delta) x^+}\big) 
	\big(y^- + \rme^{-\sqrt{2\lambda}(1-\delta) y^+}\big)}{t}
\rme^{\frac{(y-x)^2}{2t}} \,.
\end{multlined}
\end{equation}
\end{lem}
\begin{proof}
Without loss of generality, we may assume that $x \geq 0$ and $y \leq 0$. Since the family $(Y_u)_u$ have uniform upper tails, there exists $x_0 \geq 1$ such that for all $u\geq 0$ and $x \geq x_0$ 
\begin{equation}
\label{e:A100}
\bbP \big(Y_u \geq x \tfrac{\delta}{2} - 2 \delta^{-1} x^{1/2-\delta} \big) \leq \tfrac{\delta}{2}. 
\end{equation}
Suppose first that $x \in \big [x_0, (\log^+ t)^2 \big]$ and set $\tau := \inf \{s \geq 0 :\: W_s = x \tfrac{\delta}{2} \}$ and 
\begin{equation}
\label{e:674}
\cD_t(u_1, u_2) := \Big\{ \max_{k :\: \sigma_k \in [u_1, u_2]} \big( W_{\sigma_k} - \delta^{-1}  (\wedge^t(\sigma_k))^{1/2-\delta} - Y_{\sigma_k} \big) \leq 0 \Big\} \,,
\end{equation}
for $0 \leq u_1 < u_2 \leq t$. Then by conditioning on $\tau$ we may bound the probability $\bbP_{0,x}  \big( \cD_t(0,t), \, W_t \in \rmd y \big)$ by

\begin{equation}\label{e:A101}
\begin{split}
\bbP_{0,x} & \big( \tau \geq x^2 ; \, \cD_t\big(0,x^2 \big)\big) 
	\textstyle	\sup\limits_{x' \geq x\delta/2 } \bbP_{x^2, x'} \big(\cD_t\big(x^2, t\big) ; \, W_t \in \rmd y \big) \\
& + \int_{s=0}^{x^2} \bbP \big(\tau \in \rmd s ; \, \cD_t(0,s) \big) \big)\bbP_{s, x\delta/2 } \big( \cD_t(s,t) ; \, W_t \in \rmd y \big)  
\end{split}
\end{equation}


Now, if $\tau \geq s$ then by definition,
\begin{equation}\label{li.E3}
W_u - \delta^{-1}  (\wedge^t(u))^{1/2-\delta} \geq x \tfrac{\delta}{2} - 2\delta^{-1}  x^{1/2-\delta}, \quad \mbox{for all $u \leq s$}.
\end{equation}
By independence between $W$ and $Y$, the first term in the integrand is bounded above by
\begin{equation}
\bbP (\tau \in \rmd s) \bbP \Big(\min_{k :\: \sigma_k \in [0,s]}
Y_{\sigma_k} \geq x \tfrac{\delta}{2} - 2 \delta^{-1} x^{1/2-\delta} \Big)
\leq \bbP (\tau \in \rmd s) \rme^{-\lambda s \big(1- \frac{\delta}{2} \big)} \,,
\end{equation}
where we have used~\eqref{e:A100} and the total probability formula, to bound the second probability on the left hand side by $\bbE (\delta/2)^{\cN([0,s])}$. At the same time, the Gaussian density implies that
\begin{equation}
\bbP_{0,x} ( \tau \in \rmd s) / \rmd s \leq \bbP_{0,x} \big( ( 2 \delta^{-1} )W_s \in \rmd x \big) / \rmd x \leq C s^{-1/2} \exp\Big(-\tfrac{x^2(1-\delta/2)^2 }{2s} \Big) \,.
\end{equation}
Similar reasoning shows that the first term in~\eqref{e:A101} is bounded up to multiplicative constants by $ \rme^{-\lambda x^2(1-\delta / 2)}$.

Finally, from concavity it follows that $(\wedge^t(u))^{1/2-\delta} \leq (\wedge^{t-s}(u-s))^{1/2-\delta} + s^{1/2-\delta}$ for all $0 \leq s \leq u \leq t$. Using this plus stochastic monotonicity of the trajectories of $W$ in the boundary conditions to replace $x\delta /2$ by $0$, the standard Gaussian density formula and shift invariance for both $W$ and $\cN$ the last term in the integrand of~\eqref{e:A101} is bounded above by $C(t-s)^{-1/2}$ times
\begin{equation}
\bbP_{0,0}^{t-s, y} \Big( \max_{k :\: \sigma_k \in [0, t-s]} \big( W_{\sigma_k} -\delta^{-1} ( \wedge^{t-s}(\sigma_k))^{1/2-\delta} - \delta^{-1} s^{1/2-\delta} \big) - Y_{s+\sigma_k}  \big) \leq 0 \Big).
\end{equation}  
Thanks to the previous lemma and since $s \leq x^2$, in the range of the integral, the integrand in~\eqref{e:A101} is further bounded above by
\begin{equation}
C \, t^{-3/2} (\delta^{-1} s^{1/2-\delta} + 1)(y^- + \delta^{-1} s^{1/2-\delta} + 1) \leq C \,  t^{-3/2} (x^2 + 1) (y^- + 1) \,.
\end{equation}
For similar reasons, the last expression can also be used to bound the supremum in~\eqref{e:A101}.
Putting all these bounds together, display~\eqref{e:A101} is bounded above by
\begin{equation}
\label{e:A200}
C \, t^{-3/2} (x^2+1) (y^- + 1) \Big(
\rme^{-\lambda x^2(1-\delta / 2)} + \int_{s=0}^\infty \rme^{-\lambda s (1-\delta / 2)} s^{-1/2} \rme^{-\frac{(x(1-\delta/2))^2}{2s}} \rmd s \Big) \,.
\end{equation}
The exponent in the integrand is maximized at 
$s=x(1-\tfrac{\delta}{2}) \big(2 \lambda (1-\tfrac{\delta}{2})\big)^{-\frac{1}{2}}$ making the integral bounded above by 
\begin{equation}
C \exp \big(-\sqrt{2 \lambda (1-\delta/2)}(1-\delta/2) x \big) \leq C \exp \big(-\sqrt{2 \lambda} (1-\delta) x \big).
\end{equation} 
Consequently, increasing $x_0$ if necessary, the probability $\bbP_{0,x}  \big( \cD_t(0,t), \, W_t \in \rmd y \big)$ is bounded by
\begin{equation}
\label{e:A202}
C \, t^{-3/2} (y^-+1) \rme^{-\sqrt{2\lambda}(1-2\delta) x} ,
\qquad \text{for all $x \in \big[x_0, (\log^+ t)^2 \big]$}.
\end{equation}
If $x > (\log^+ t)^2$ then noting that $\bbP(\tau > t,\, W_t \in \rmd y) = 0$, we bound the probability  $\bbP_{0,x}  \big( \cD_t(0,t), \, W_t \in \rmd y \big)$ by
\begin{equation}
\int_{s=0}^{t} \bbP \big(\tau \in \rmd s \,,\,\, \cD_t(0,s) \big)
\end{equation}
Proceeding as before, the integral is bounded above by $\rme^{-\sqrt{2\lambda}(1-\delta) x}$, which can be made smaller than~\eqref{e:A202} for all $t$ and $x > (\log^+ t)^2$, by modifying the constant in~\eqref{e:A202} appropriately. Dividing this bound by 
$\bbP_{0,x}(W_t \in \rmd y) = (2 \pi t)^{-1/2} \rme^{-\frac{(y-x)^2}{2t}}$ gives~\eqref{e:A10} for the case $x > x_0$. 

For $x < x_0$, we use stochastic monotonicity of $W$ in the boundary conditions under the conditional measure, to replace $x$ with $0$. This can only increase the probability in question. Then we apply the previous Lemma \ref{l:A1} to bound the left hand side of~\eqref{e:A10} by $C t^{-1}(y^-+1)$. Increasing $C$ if necessary, this can be made smaller than the right hand side of~\eqref{e:A10} for all $x \in [0, x_0]$, $y \leq 0$ and $t \geq 0$. The result follows. 
\end{proof}
The proof of Proposition~\ref{p:A2} is now straightforward.
\begin{proof}[Proof of Proposition~\ref{p:A2}]
Starting with the first upper bound, the probability on the left hand side of~\eqref{e:2.11} can be written as 
\begin{equation}
\label{e:9.1}
\bbP_{0,x'}^{t,y'} \Big( \max_{k :\:\sigma_k \in [0,t]} \big(W_{\sigma_k} - \delta^{-1}  (\wedge^t(\sigma_k))^{1/2-\delta} - Y_{\sigma_k}\big) \leq 0 \Big) \,.
\end{equation}
with $x' := x - \delta^{-1}$ and $y' := y-\delta^{-1}$. Then, thanks to 
Lemma~\ref{l:A1}, the last display is at most
\begin{equation}
\begin{split}
C\frac{(x'^-+1)(y'^-+1)}{t} & \leq C\frac{(x^-+\delta^{-1} + 1)(y^- +\delta^{-1} + 1)}{t} \\
& \leq  C \delta^{-2} \frac{(x^-+1)(y^-+1)}{t} \,,
\end{split}
\end{equation}
which proves the first statement in the proposition.

To show~\eqref{e:2.12}, we proceed similarly, using this time Lemma~\ref{l:A2} to bound the probability in~\eqref{e:9.1} by the right hand side of~\eqref{e:A10} with $x',y'$ replacing $x,y$. It is not difficult to see that the latter bound can be converted to the form appearing in the right hand side of~\eqref{e:2.12} by adjusting the constant appropriately.
\end{proof}

\subsection{Entropic Repulsion and Asymptotics}
The upper bound in Proposition~\ref{p:A2} can be used to derive the following lemmas concerning the entropic repulsion effect. The first lemma is rather sharp and  hence requires additional regularity conditions on the curve. This will be used directly to derive Proposition~\ref{p:A5}. The second lemma is much coarser and does not require any additional conditions. This lemma will be used as a key input to Proposition~\ref{p:A3}. As the arguments of both are similar, we shall only prove the first and harder of the two. 

For what follows, let us define for fixed $\epsilon > 0$ and all $t \geq 0$ the set,
 \begin{equation}
R_\epsilon(t) := \big\{(x,y) :\: x,y \leq 1/\epsilon \,,\  (x^-+1)(y^-+1) \leq t^{1-\epsilon} \big\} \,,
\end{equation}
corresponding to the range at which uniformity may be obtained. We shall also need the events
\begin{equation}
\label{e:1074a}
\cG_t(u_1, u_2) := \Big\{ \max_{k :\: \sigma_k \in [u_1, u_2]} \big( W_{\sigma_k} - \delta^{-1}  (1+\wedge^t(\sigma_k))^{1/2-\delta} - Y_{\sigma_k} \big) \leq 0 \Big\} \,,
\end{equation}
and
\begin{equation}
\label{e:1074}
\cQ_t(u_1, u_2) := \Big\{ \max_{k:\: \sigma_k \in [u_1, u_2]} \big( W_{\sigma_k} - \gamma_{t,\sigma_k} - Y_{\sigma_k} \big) \leq 0 \Big\} \,, 
\end{equation}
for $0 \leq u_1 < u_2 \leq t$. 

\begin{lem}
\label{l:10.7.0}
Let $\delta \in (0,1/2), \lambda > 0$ be given and suppose that $W, Y, \cN$ and $\gamma$ are defined as in the Subsection~\ref{ss:Setup} and that in particular~\eqref{e:A1}, ~\eqref{e:A2},~\eqref{e:A3} and~\eqref{e:A3s} hold. Then for any $\epsilon > 0$ there exists $C = C(\lambda, \delta, \epsilon)$ such that for all $M > 0$, $2 < s \leq t/2$ and $x,y \in R_\epsilon(t)$,
\begin{equation}
\bbP_{0,x}^{t,y} \Big( \max_{u \in [s, t-s]} \big(W_u - \gamma_{t,u}\big) \geq -M ,\, \cQ_t (0,t) \Big) \leq C \frac{(x^-+1)(y^-+1)}{t} \frac{(M+1)^2}{\sqrt{s}} \,.
\end{equation}
\end{lem}
\begin{proof}
We shall show that for all $r \in [s-1, t-s]$,
\begin{equation}
\label{e:M7.34}
\bbP_{0,x}^{t,y} \Big(\cQ_t(0,t) ,\, \max_{q \in [r, r+1]} (W_q - \gamma_{t,q}) \geq -M) 
\leq C \frac{(x^-+1)(y^-+1)}{t(r \wedge (t-r))^{3/2}} (M+1)^2 \,.
\end{equation}
Then summing both sides from $r=\lfloor s \rfloor $ to $r=\lceil t-s \rceil - 1$ and using the union bound will yield the desired statement.

To this end, for any $r \in [s-1 \,,\, t/2-1]$ and $z \in \bbR$, we may write
\begin{equation}
\label{e:493}
\bbP_{0,x}^{t,y} \big(\cQ_t(0,t) ,\, W_r \in \rmd z\big)
= \bbP_{0,x}^{r,z} \big(\cQ_t(0,r) \big) 
\times \bbP_{r,z}^{t,y} \big(\cQ_t(r,t) \big) 
\times \bbP_{0,x}^{t,y} \big(W_r \in \rmd z \big) \,.
\end{equation}
Since $W_u$ is Gaussian, we may replace it by $W_u +  \frac{u}{r} \gamma_{t,r}$ everywhere in the first probability on the right hand side (including the conditioning). Then we use the first part of Condition~\eqref{e:A3s} to obtain
\begin{equation}
\label{e:195B}
\bbP_{0,x}^{r,z} \big(\cQ_t(0,r) \big) \leq \bbP_{0,x}^{r,z'} \big(\cG_r(0,r)\big) \,,
\end{equation}
where $z' = z - \gamma_{t,r}$.
Similarly, replacing $W_u$ by $W_u + \gamma_{t,r}(t-u)/(t-r)$, using the shift invariance of $W$ and $\cN$ and the second part of condition~\eqref{e:A3s} show that 
\begin{equation}
\label{e:196B}
\bbP_{r,z}^{t,y} \big(\cQ_t(r,t) \big) \leq \bbP_{0,z'}^{t-r,y} \big(\cG_{t-r}(0,t-r)\big) \,,
\end{equation}
with $z'$ as before.
Now, if $z' \leq 0$ we can use the first upper bound in Proposition~\ref{p:A2} to bound the product of both probabilities above by,
\begin{equation}
\label{e:196BA}
C \frac{(x^-+1)(z'^-+1)^2(y^-+1)}{r(t-r)} \,.
\end{equation}
If $z' > 0$ and $x \leq y$, we use the first upper bound in Proposition~\ref{p:A2} for the right hand side of~\eqref{e:195B} and the second upper bound in the proposition for the right hand side of~\eqref{e:196B}. In this case the product of the two probabilities is upper bounded by
\begin{equation}
\label{e:197B}
C \frac{(x^-+1)(y^-+1)}{r(t-r)} \rme^{-Cz'} \exp \Big(\tfrac{(z'+y^-)^2}{2(t-r)}\Big) \,.
\end{equation}
Above, we have also used the stochastic monotonicity of $W$ with respect to the initial conditions, to replace $y$ by $y \wedge 0 = -y^-$ before applying the upper bound. If $z' > 0$ but $x > y$, we use the bounds in the opposite way, thereby obtaining the bound
\begin{equation}
\label{e:198B}
C \frac{(x^-+1)(y^-+1)}{r(t-r)} \rme^{-Cz'} \exp \Big(\tfrac{(z'+x^-)^2}{2r} \Big) \,,
\end{equation}
on the product of the probabilities.

Next we bound $\bbP_{0,x}^{t,y} (W_r \in \rmd z)$.
Since $W_r$ under $\bbP_{0,x}^{t,y}$ is Gaussian with variance $r(t-r)/t \in [\wedge^t(r)/2,\, \wedge^t(r)]$ and mean $t^{-1} \big(xr + y (t-r)\big)$, we can bound
\begin{equation}
\frac{\bbP_{0,x}^{t,y}(W_r \in \rmd z)}{\rmd z}
\leq \begin{cases}
	C r^{-1/2} &  \text{if } z' \leq 0 \,, \\
	C r^{-1/2} \exp \Big(-\tfrac{(z-(x \vee y))^2}{2\wedge^t(r)} \Big)
		& \text{if } z' > 0 \,.
\end{cases}
\end{equation}
where $C$ may depend on $\epsilon$. Multiplying the this by either~\eqref{e:196BA},~\eqref{e:197B} or~\eqref{e:198B} as prescribed above, we obtain
\begin{equation}
\bbP_{0,x}^{t,y} \big(\cQ_t(0,t) ,\, W_r \in \rmd z\big) / \rmd z \leq
C \frac{(x^-+1)(y^-+1)}{t r^{3/2}} (z'^-+1)^2 \rme^{-Cz'^+} \,,
\end{equation}
which is valid for all $z'$. Finally, integrating the left hand side above from $\gamma_{t,r}-M$, with $M \geq 1$ to $\infty$ gives
\begin{equation}
\label{e:M7.42}
\bbP_{0,x}^{t,y} \big( W_r \geq \gamma_{t,r} -M \,, \cQ_t(0,t)\big) \leq C
\frac{(x^-+1)(y^-+1)}{tr^{3/2}} (M+1)^2 \,.
\end{equation}

In a similar way, if $r \in [s-1 \,,\, t/2-1]$ and $z,w \in \bbR$ satisfy $z \leq \gamma_{t,r}$ and $w \leq \gamma_{t,r+1}$, we can bound above $\bbP_{0,x}^{t,y} \big(\cQ_t(0,t) \,,\,\, W_r \in \rmd z\,,\,\, W_{r+1} \in \rmd w\big)$ by
\begin{equation}\begin{multlined}
\bbP_{0,x}^{r,z} \big(\cQ_t(0,r) \big) 
\times \bbP_{r+1,w}^{t,y} \big(\cQ_t(r+1,t) \big) 
\times \bbP_{0,x}^{t,y} \big(W_r \in \rmd z,\, W_{r+1} \in \rmd w \big) \\
\leq C \frac{(x^-+1)(y^-+1)(z'^-+1)(w'^-+1) \rme^{-C'(z-w)^2}}{tr^{3/2}} \rmd z \rmd w\,,
\end{multlined}\end{equation}
where we let $z' = z- \gamma_{t,r}$, $w' = w - \gamma_{t,r+1}$ and bounded the probability on the left hand side by $Cr^{-1/2} \rme^{-C(z-w)^2}$ using the assumptions on $x$ and $y$. At the same time, it follows from~\eqref{e:A3s} that $|\gamma_{t,r} - \gamma_{t,q}| \leq C$ for all $1 \leq r \leq q \leq r+1 \leq t-1$ for some $C = C(\delta)$. Therefore, for $r,z,w$ as above, the reflection principle for Brownian motion together with stochastic monotonicity imply
\begin{equation}
\bbP_{r,z}^{r+1,w} \Big(\max_{q \in[r,r+1]} W_q - \gamma_{t,q} \geq -M \Big)
\leq \bbP_{0, z' \vee w'}^{1, z' \vee w'} \big( \max_{q \in [0,1]} W_q \geq -M - C \big) \,,
\end{equation}
which is at most $C \rme^{-C' (-M -z' \vee w')^2}$. Using the last two bounds, we may write for all $r \in [s, t/2]$,
\begin{equation}
\begin{split}
\label{e:M7.45}
\bbP_{0,x}^{t,y} &\Big(\cQ_t(0,t) \,,\,\, \max_{q \in [r, r+1]} (W_q - \gamma_{t,q}) \geq -M \,,\,\,
\max_{k=0,1} \big(W_{r+k} -\gamma_{t,r+k}\big) \leq -M  \Big) \\
& \leq \int  
\bbP_{0,x}^{t,y} \big(\cQ_t(0,t) ,\, W_r \in \rmd z, W_{r+1} \in \rmd w\big) \\
& \qquad \qquad \bbP_{r,z}^{r+1,w} \Big(\max_{q \in [r, r+1]} \big(W_q - \gamma_{t,q}\big) \geq -M \Big) \\
& \leq C \frac{(x^-\!+1)(y^-\!+1)}{tr^{3/2}}
\int (z'^-\!+1)(w'^-\!+1) \rme^{-C' ((z'-w')^2 + (M + z' \vee w')^2)} \rmd z \rmd w  \\
& \leq C \frac{(x^-+1)(y^-+1)}{tr^{3/2}}(M+1)^2 \,,
\end{split}
\end{equation}
where $z \leq \gamma_{t,r} - M$, $w \leq \gamma_{t,r+1}-M$ is the range of integration in both integrals and we have also used that $|z-w| > |z'-w'|-C$ which follows from the bound on the difference $|\gamma_{t,r}-\gamma_{t,r+1}|$ as shown above.

Summing the last display together with display~\eqref{e:M7.42} twice -- for $r$ and $r+1$, and using the union bound, we get~\eqref{e:M7.34} for $r \in [s-1 \,,\, t/2-1]$. A symmetric argument will show the same for $r \in [t/2 \,,\, t-s]$.
\end{proof}

The second entropic repulsion result is provided in the following lemma. 
\begin{lem}
\label{l:10.7}
Let $\delta \in (0,1/2), \lambda > 0$ be given and suppose that $W, Y, \cN$ and $\gamma$ are defined as in the Subsection~\ref{ss:Setup} and that in particular~\eqref{e:A1} and~\eqref{e:A2} hold. Then for any $\epsilon > 0$ and $M > 0$, there exists a positive constant $C = C(\lambda, \delta, \epsilon, M)$ such that for all $2 < s \leq t/2$ and $x,y \in R_\epsilon(t)$,
\begin{multline}
\nonumber
\bbP_{0,x}^{t,y} \bigg( 
\! \Big\{ \textstyle \max\limits_{\sigma_k \in [0,t]} 
	W_{\sigma_k} - \delta^{-1} \big(1+ (\wedge^t(\sigma_k))^{1/2-\delta}\big) - Y^+_{\sigma_k} \leq 0 \Big\} \, \cap  \\
 \Big(\big\{ \max\limits_{[s,t-s]} W_u \geq 0 \big\} 
\cup 
	\big\{\max\limits_{\sigma_k \in [s,t-s]} W_{\sigma_k} \! + \! Y_{\sigma_k}^- \geq -M \big\} \Big) \bigg) 
		\displaystyle \leq C \frac{(x^- \! + \! 1)(y^- \! + \! 1)}{t s^{1/4}}\,.
\end{multline}
\end{lem}
\begin{proof}
The proof uses similar arguments as in the proof of the previous lemma. We therefore omit it.
\end{proof}

We now turn to the asymptotic statements in Subsection~\ref{ss:Setup}, namely Propositions~\ref{p:A3} and~\ref{p:A4}. In analog to $\cQ_t$ from~\eqref{e:1074} let us define for $0 \leq u_1 < u_2 < t$ the event,
\begin{equation}
\label{e:1074.1}
\cQ^-_t(u_1, u_2) := \Big\{ \max_{\sigma_k \in [u_1, u_2]} \big( W_{\sigma_k} - \gamma_{t,t-\sigma_k} - Y_{t-\sigma_k} \big) \leq 0 \Big\} \,.
\end{equation}
We shall also need the $t=\infty$ analogs of $\cQ_t(u_1, u_2)$ and $\cQ_t^-(u_1, u_2)$, which we define as
\begin{equation}
\label{e:692}
\begin{split}
\cQ_\infty & (u_1, u_2)  := \Big\{ \max_{\sigma_k \in [u_1, u_2]} \big( W_{\sigma_k} - \gamma_{\infty,\sigma_k} - Y_{\sigma_k} \big) \leq 0 \Big\} ; \\
& \text{and} \quad
\cQ^-_\infty(u_1, u_2) := \Big\{ \max_{\sigma_k \in [u_1, u_2]} \big( W_{\sigma_k} - \gamma_{\infty,-\sigma_k} - Y_{\infty}^{(\sigma_k)} \big) \leq 0 \Big\} \,,
\end{split}
\end{equation}
where $0 \leq u_1 < u_2$ and $(Y_\infty^{(u)} :\: u \geq 0)$ is an independent collection of random variables with $Y^{(u)}_\infty \overset{law}= Y_\infty$, which we assume to be defined on the same probability space as $W, \cN, Y$ and independent of $W$ and $\cN$. Finally, we shall also need
\begin{equation}
\label{e:795}
f_s(x) := \bbE_{0,x} \big(W_s^- ;\; \cQ_\infty(0,s)\big) \ ,\ \ 
g_s(y) := \bbE_{0,y} \big(W_s^- ;\; \cQ^-_\infty(0,s)\big) \,,
 \end{equation}
for any $s \geq 0$ and $x,y \in \bbR$.
The proof of Proposition~\ref{p:A3} will rely on the following lemmas. 
\begin{lem}
\label{l:10.9}
Let $\delta \in (0,1/2), \lambda >0$ and suppose that $W, Y, \cN$ and $\gamma$ are defined as in Subsection~\ref{ss:Setup} and in particular that~\eqref{e:A1},~\eqref{e:A2},~\eqref{e:A3},~\eqref{e:A2.5} and~\eqref{e:A4} hold. Then for any $\epsilon > 0$,
\begin{equation}
\label{e:196}
\lim_{s \to \infty} \limsup_{t \to \infty} \sup_{x,y \in R_\epsilon(t)}
\frac{t}{(x^-+1)(y^-+1)}
\left|\bbP_{0,x}^{t,y} \big( \cQ_t(0,t) \big)
- \frac{2f_s(x) g_s(y)}{t} \right| = 0 \,.
\end{equation}
\end{lem}

\begin{lem}
\label{l:10.10}
Let $\delta \in (0,1/2), \lambda >0$ and suppose that $W, Y, \cN$ and $\gamma$ are defined as in Subsection~\ref{ss:Setup} and in particular that~\eqref{e:A1},~\eqref{e:A2},~\eqref{e:A3},~\eqref{e:A2.5} and~\eqref{e:A4} hold. Then for all $x,y \in \bbR$,
\begin{equation}
\label{e:598}
\liminf_{s \to \infty} f_s(x) > 0 \quad , \qquad
\liminf_{s \to \infty} g_s(y) > 0 \,.
\end{equation}
Moreover for each $s \geq 1$,
\begin{equation}
\label{e:599}
\lim_{x \to \infty} \frac{f_s(-x)}{x} = \lim_{y \to \infty} \frac{g_s(y)}{y} = 1 \,.
\end{equation}
\end{lem}

Before proving these lemmas, let us use them to prove Proposition~\ref{p:A3}.
\begin{proof}[Proof of Proposition~\ref{p:A3}]
Fix $\epsilon > 0$ and let us first take also fixed $x,y < 1/\epsilon$. Then using Lemma~\ref{l:10.9} and distributing the $t$ factor in~\eqref{e:196}, the absolute value becomes the difference between functions: $t \mapsto t \bbP_{0,x}^{t,y} \big(\cQ_t(0,t)\big)$ and $s \mapsto 2f_s(x)g_s(y)$. Then the vanishing of the limit in~\eqref{e:196} shows that in fact both functions are Cauchy in their respective arguments and hence converging to an identical limit:
\begin{equation}
\lim_{t \to \infty} t\bbP_{0,x}^{t,y} \big(\cQ_t(0,t)\big) = 
\lim_{s \to \infty} 2 f_s(x) g_s(y) =: 2 h_\infty(x,y) < \infty \,,
\end{equation}
which is finite (but possibly zero). 

Taking $Y'_u :\overset{law}=Y_\infty$ and $\gamma'_{t,u} := \gamma_{\infty, -\wedge^t(u)}$ for all $0 \leq u \leq t$, Conditions~\eqref{e:A2.5},~\eqref{e:A3} and~\eqref{e:A4} still hold with $Y'_\infty \overset{law} = Y_\infty$ and $\gamma'_{\infty, u} = \gamma'_{\infty, -u} = \gamma_{\infty, -u}$. Moreover, denoting by $f'_s$ and $g'_s$ the corresponding analogs of $f_s$ and $g_s$ from~\eqref{e:795}, we have $f'_s(y) = g'_s(y) = g_s(y)$ for all $s \geq 0$ and $y < 1/\epsilon$. Since Lemma~\ref{l:10.9} is still in force with $Y'$ and $\gamma'$ in place of $Y$ and $\gamma$, the following limit exists for all $y < 1/\epsilon$
\begin{equation}
\lim_{s \to \infty} f'_s(y) g'_s(y) = \lim_{s \to \infty} g_s(y)^2 
\end{equation}
It follows that $g(y) := \lim_{s \to \infty} g_s(y)$ exists and is finite for all such $y$ and in light of the first part of Lemma~\ref{l:10.10} also positive. Returning to our original $Y$ and $\gamma$, this also shows that $f(x) := \lim_{s \to \infty} f_s(x)$ exists for all $x < 1/\epsilon$, that it is finite and positive (by Lemma~\ref{l:10.10} again) and moreover, that $h_\infty(x,y) = f(x) g(y)$.

Turning to the issue of uniformity, the fact that the limit in~\eqref{e:196} is uniform with respect to $x,y \in R_\epsilon(t)$ in~\eqref{e:196} shows that 
\begin{equation}
\label{e:301}
\big| f_s(x)g_s(y) - f(x) g(y)\big| = (x^-+1)(y^-+1) o(1)
\quad \text{as } s \to \infty \,,
\end{equation}
uniformly in all $x,y < 1/\epsilon$ and
\begin{equation}
\label{e:302}
\Big| \bbP_{0,x}^{t,y} \big(\cQ_t(0,t)\big) - \frac{2 f(x) g(y)}{t} \Big| = 
\frac{(x^-+1)(y^-+1)}{t} o(1) 
\quad \text{as } t \to \infty \,,
\end{equation}
uniformly in $x,y \in R_\epsilon(t)$. 
Using~\eqref{e:301} with $y=0$, it follows from the uniformity of the $o(1)$ term that
\begin{equation}
\lim_{s \to \infty}\limsup_{x \to -\infty} \bigg|\frac{f_s(x)}{x^-+1}g_s(0) - \frac{f(x)}{x^-+1}g(0)\bigg| = 0 \,.
\end{equation}
Using this together with the second part of Lemma~\ref{l:10.10} and a similar argument with the roles of $f$ and $g$ reversed, yields
\begin{equation}
\label{e:204}
\lim_{x \to \infty} f(-x)/x = \lim_{y \to \infty} g(-y)/y = 1 \,.
\end{equation}

Now, together with the positivity of $f(x)$ and $g(y)$ the latter implies that there exists $C > 0$ such that $f(x) > C(x^-+1)$ and $g(y) > C(y^-+1)$ for all $x,y < 1/\epsilon$. In turn, we then have that the right hand side of~\eqref{e:302} is actually $o\big(f(x)g(y)/t\big)$ uniformly in $x,y \in R_\epsilon(t)$ as $t \to \infty$. This shows~\eqref{e:A6} with $f$ and $g$, while~\eqref{e:204} is exactly~\eqref{e:A6.1}.
\end{proof}

It remains therefore to prove Lemma~\ref{l:10.9} and Lemma~\ref{l:10.10}. 
\begin{proof}[Proof of Lemma~\ref{l:10.9}] 
Given $s \geq 1$ we observe that for any $t \geq 2s$,
\begin{equation}\begin{multlined}
\cQ_t(0,t)\, \triangle\, \Big( \cQ_t(0,s) \cap \cQ_t(t-s,t) \cap \Big\{\max\limits_{[s,t-s]} W_u \leq 0 \Big\}\Big) \\
\subseteq \Big\{ \max_{\sigma_k \in [0,t]} 
	W_{\sigma_k} - \delta^{-1} \big(1+(\wedge^t(\sigma_k))^{1/2-\delta}\big) - Y^+_{\sigma_k} \leq 0 \Big\} \cap 
\Big(\Big\{\max_{[s,t-s]} W_u \geq 0 \Big\} \\ \cup 
	\Big\{\max_{\sigma_k \in [s,t-s]} W_{\sigma_k} + Y_{\sigma_k}^- \geq -\delta^{-1} \Big\} \Big) \,,
\end{multlined}\end{equation}
where we have used the bounds on $\gamma_{t,u}$ per~\eqref{e:A3}. But then (the entropic repulsion) Lemma~\ref{l:10.7} with $M = \delta^{-1}$ yields
\begin{equation}
\label{e:A16}
\begin{multlined}
\textstyle
\limsup\limits_{\substack{t \to \infty \\ s \to \infty}}  \!  \sup\limits_{x,y \leq \epsilon^{-1}} \, \frac{t}{(x^-+1)(y^-+1)}  \! 
\Big| 
\textstyle \bbP_{0,x}^{t,y} \Big(\max\limits_{k :\: \sigma_k \in [0,t]} W_{\sigma_k} - \gamma_{t,\sigma_k} - Y_{\sigma_k} \leq 0 \Big) \\
\textstyle
-\bbP_{0,x}^{t,y} \Big( \max\limits_{k :\: \wedge^t(\sigma_k) \leq s} W_{\sigma_k} \! - \! \gamma_{t,\sigma_k} \! - \! Y_{\sigma_k} \leq 0 
	\,,\,\, \max\limits_{[s,t-s]} W_u \leq 0\Big)\Big| = 0,
\end{multlined}
\end{equation}
with the order in the limits taken from top to bottom.

To estimate the second probability above, we claim that we can add the restriction $\{|W_s-x| \vee |W_{t-s} - y| < \log^+ t\}$ to the corresponding event at a uniform $o(t^{-1})$ cost. Indeed, $W_s$ and $W_{t-s}$ are Gaussian under $\bbP_{0,x}^{t,y}$ with variance $s(t-s)/t \leq s$ and means $x+(y-x)s/t$ and $y+(x-y)s/t$ respectively. Since $(y-x)/t \to 0$ as $t \to \infty$, uniformly
in $x,y \in R_\epsilon(t)$, it follows from the standard Gaussian tail estimate and union bound that
\begin{equation}
\bbP_{0,x}^{t,y} \big( |W_s-x| \vee |W_{t-s} - y| > \log^+t \big)
\leq C \exp\Big(-\tfrac{\log^2 t}{s} \Big) = o (t^{-1})
\end{equation}
as $t \to \infty$, uniformly in $x,y \in R_\epsilon(t)$ for fixed $s \geq 0$.

If $z,w \leq 0$ satisfy $|z-x| \vee |w-y| < \log t$, then $z^- w^- \leq t^{1-\epsilon/2}$ for all $t$ large enough, so we can use the reflection principle for standard Brownian motion to obtain
\begin{equation}
\nonumber
\bbP_{0,x}^{t,y} \big( \max_{[s,t-s]} W_u \leq 0 \,\big|\, W_s = z, W_{t-s} = w \big)
= 1-\exp \Big(-\tfrac{2(z^- w^-)}{t-2s}\Big) 
\end{equation}
which is asymptotic to $2\frac{z^- w^-}{t}$ uniformly in $x,y \in R_\epsilon(t)$ for fixed $s$. Also in this range of $z$ and $w$ the joint density $\bbP_{0,x}^{t,y} \big(W_s \in \rmd z,\, W_{t-s} \in \rmd w \big)/\rmd z \rmd w$ is equal to
\begin{equation}
\frac{\exp \left(-\frac{(z-x)^{2} -(y-w)^2}{2s}\right)} 
{2 \pi s^2 }
\times \sqrt{\frac{t}{(t-s)}}
\exp \left( -\tfrac{(w-z)^{2}}{2(t-s)} + \tfrac{(x-y)^{2}}{2 t} \right),
\end{equation}
with the term after the product sign tending to $1$ as $t \to \infty$ uniformly as above.

Conditioning on $(W_s, W_{t-s}) = (z,w)$ and using the total probability formula, the second probability in~\eqref{e:A16}, restricted by $\{|W_s - x| \vee |W_{t-s}-y| < \log t\}$ is therefore equal to
\begin{multline}
\label{e:399}
\frac{2}{t} \bbE_{0,x} \Big( W_s^- ;\  \cQ_t(0,s) ,\, |W_s - x| < \log t \Big) \\
\times \bbE_{0,y} \Big( W_s^- ;\  \cQ^-_t(0,s),\, |W_s - y| < \log t \Big) \big(1+o(1) \big) \,,
\end{multline}
where 
we have also used the shift and time-reversal invariance in law of both $W$ and $\cN$.

Next, we wish to get rid of the restriction on $|W_s-x|$ and $|W_s -y|$. To this end, we note that $\bbE_{0,x} (W_s^-) \leq \bbE_{0,0} (W_s^-) + x^- \leq s^{1/2} +  x^-$ and thanks to the Gaussian tail also
\begin{equation}
\nonumber
\bbE_{0,x} \big(W_s^-;\; |W_s - x| > \log t \big) \!\leq x^- \bbP_{0,0} \big(|W_s| > \log t\big) + 
	\bbE_{0,0} \big(|W_s|;\; |W_s| > \log t\big)
\end{equation}
which tends to $0$ as $t \to \infty$ uniformly in $x$ in its allowed range. Since the same clearly holds for $y$, we may remove the events $\{|W_s - x| < \log t\}$ and $\{|W_s - y| < \log t\}$ 
from the expectations in~\eqref{e:399} as well as the $1+o(1)$ factor in the end of the display, 
at the cost of an extra term which is $o((x^-+1)(y^-+1)/t)$ as $t \to \infty$.

It remains to show that both expectations converge uniformly to $f_s(x)$ and $f_s^-(y)$ respectively when $t \to \infty$. Assumptions~\eqref{e:A2.5} and~\eqref{e:A4} and the absolute continuity of $W_{\sigma_k}$ with respect to the Lebesgue measure could be used in conjunction with the dominated convergence theorem to show that this is indeed the case for fixed $x,y$. However, in order to show that the convergence is uniform, we need to work slightly harder. We proceed detailing the argument for the convergence of the second expectation to $g_s(y)$. The argument for the convergence of the first to $f_s(x)$ is similar and in fact easier, since the random variables $Y_{\sigma_k}$ do not depend on $t$.

The weak convergence $Y_u \Longrightarrow Y_\infty$ as $u \to \infty$, imply via a standard coupling argument, that we can assume that collections $(Y_u :\: u \geq 0)$ and $(Y^{(u)}_\infty :\: u \geq 0)$ are defined jointly in such a way that $Y_u - Y^{(u)}_\infty $ converges to $0$ in probability. Letting now $\wt{\cQ}^-_\infty(0,s)$ be defined as $\cQ^-_\infty(0,s)$ in~\eqref{e:692} only with $Y^{(t-\sigma_k)}_\infty$ in place of $Y^{(\sigma_k)}_\infty$ and observing that the definition of $f_s^-(y)$ does not change if we use $\wt{\cQ}^-_\infty(0,s)$ in place of $\cQ^-_\infty(0,s)$, we can write
\begin{multline}
\label{e:605}
\big| \bbE_{0,y} \big( W_s^- ;\  \cQ^-_t(0,s) \big) - g_s(y) \big|
\leq \bbE_{0,y} \big( W_s^- ; \ \cQ^-_t(0,s) \triangle \wt{\cQ}^-_\infty(0,s) \big) \\
\leq \big(\bbE_{0,y} W_s^2\big)^{1/2} \Big(\bbP_{0,y}\big(\cQ^-_t(0,s) \triangle \wt{\cQ}^-_\infty(0,s)\big)\Big)^{1/2} \,,
\end{multline}
where we have used Cauchy-Schwarz for the last inequality.

Now the first term in the second line is equal to $(y^2+s)^{1/2} \leq \sqrt{s} C (y^-+1)$ for all $y \leq \epsilon^{-1}$. For the second term in the display, define for each $\sigma_k \in  [0,s]$ the closed interval $I_{\sigma_k}$ whose endpoints are given by 
\begin{equation}
\big\{\gamma_{t,t-\sigma_k} + Y_{t-\sigma_k} \,,\,\, \gamma_{\infty, -\sigma_k}+ Y_\infty^{(t-\sigma_k)} \big\} \,.
\end{equation}
Then
$
\cQ^-_t(0,s) \triangle \wt{\cQ}^-_\infty(0,s) 
\subset \{ \exists \sigma_k \in [0,s] :\: W_{\sigma_k} \in I_{\sigma_k} \}$.  
Conditioning on $\cN([0,s]) = i$ and $Y$ the (conditional) probability of the latter event is smaller than 
$i \bbP \big(W_U \in I_U \big)$, with $U$ is an independent uniform random variable in $[0,s]$. In particular, denoting by $|I_u|$ the Lebesgue measure of $I_u$ (with $I_u$ defined as above replacing $\sigma_k$ by $u$) we have that
\begin{equation}\label{e:a.7.1}
\begin{split}
\bbP_{0,y} \big(W_U \in I_U \big) 
&\leq \frac{1}{s} \int_{0}^{s}   1 \wedge 
\bigg( |I_u| \times  \sup_{x \in \bbR} \frac{\rme^{-\frac{(x-y)^{2}}{u}}}{\sqrt{2 \pi u}} \bigg) 
 \rmd u \\
 & \leq 
\frac{C}{s} \int_{0}^{s} \big(u^{-\frac{1}{2}} \vee 1 \big) \times  \big( |I_u| \wedge 1 \big)  \rmd u.
\end{split}
\end{equation}
The length of $I_u$ is at most $|\gamma_{\infty,-u} - \gamma_{t, t-u}| + |Y_\infty^{(t-u)} - Y_{t-u}|$ and hence $|I_u| \wedge 1$ may be further upper bounded by 
$|\gamma_{\infty,-u} - \gamma_{t, t-u}| + |Y_\infty^{(t-u)} - Y_{t-u}| \wedge 1 $. 
Using the total expectation formula, the upper-bound from \eqref{e:a.7.1} and Fubini to exchange between the Lebesgue integral over $u$ and the expected value with respect to $Y$ we obtain
\begin{equation}\label{e:a.7.2}
\begin{split}
&  \bbP_{0,y}  \big( \exists \sigma_k \in \cN\cap [0,s] :\: W_{\sigma_k} \in I_{\sigma_k} \big)
\\
& \, \leq \frac{C\bbE \big[ \cN([0,s]) \big]}{s} \int_{0}^{s} \big(u^{-\frac{1}{2}} \!\vee\! 1 \big)   
\Big( |\gamma_{\infty,-u} \!-\! \gamma_{t, t-u}|\!+\! \bbE \big[|Y_\infty^{(t-u)} \!-\! Y_{t-u}| \!\wedge\! 1 \big] \Big)  \rmd u \\
&\, \leq C \lambda s \times  \sup_{u \in [0,s]}  
\Big( |\gamma_{\infty,-u} - \gamma_{t, t-u}| + \bbE \big[|Y_\infty^{(t-u)} - Y_{t-u}| \wedge 1 \big] \Big) \,. 
\end{split}
\end{equation}

Thanks to~\eqref{e:A3} and our coupling, the right hand side of~\eqref{e:a.7.2} goes to $0$ as $t$ tends to $\infty$.
The above upper bounds together with~\eqref{e:605} yield   
\begin{equation} 
\limsup_{t \to \infty} \frac{\big| \bbE_{0,y} \big( W_s^- ;\  \cQ^-_t(0,s) \big) - g_s(y) \big|}{(y^-+1)} = 0, 
\quad \text{}
\end{equation}
uniformly in $x,y \in R_\epsilon(t)$. A similar argument, with the obvious changes, applies to $f_s(x)$. Since $\bbE_{0,x}(W_s^-) \leq C(x^-+1)$ and  $\bbE_{0,y}(W_s^-) \leq C(y^-+1)$, we can further replace the product of the expectations in~\eqref{e:399} by $f_s(x) g_s(y)$ at the cost of another $o((x^-+1)(y^-+1)/t)$ uniformly for all $x,y < 1/\epsilon$. Collecting all error terms we obtain~\eqref{e:196}.

\end{proof}

\begin{proof}[Proof of Lemma~\ref{l:10.10}]
Beginning with~\eqref{e:598}, we shall prove only the first inequality, as the argument for the proof of the second is identical. Observing that $f_s(x)$ is decreasing in $x$, it is enough to consider positive $x$-s. We now fix such $x \geq 0$, let $M > 0$ and use the independence between $W$,$Y$ and $\cN$, the Markov property for $W$ to bound $f_{s+1}(x)$ from below for any $s \geq 1$ by
\begin{equation}
\label{e:215}
\begin{split}
\bbP & (\sigma_1 \geq 2) \bbP_{0,x} (W_1 \leq - \delta^{-1} -1 ) \\
&\times \bbP_{0,x} \Big( \min_{k: \sigma_k \in [2,s+1]} \big(Y_{\sigma_k} + M \log (\sigma_k-1) \big) \geq 0 \Big) \\
& \ \ \times \bbE_{1, -1} \Big(W_{s+1}^- ;\; \max_{u \in [1,s+1]} \big(W_u + M \log^+(u - 1) \big) \leq 0 \Big) \,.
\end{split}
\end{equation}
Above we have also used stochastic monotonicity of $W$ in the initial condition.

The first two terms in \eqref{e:215} are positive for any $x$. To lower bound the third, we condition on $\cN$ and use Assumption~\eqref{e:A1} and independence to bound it below by
\begin{equation}
\begin{multlined}
\bbE \exp \Big(\int_{u=2}^s \log \big(1- \delta^{-1} \rme^{-\delta M \log (u-1)} \big) \cN(\rmd u) \Big) \\
 = \exp \Big( -\lambda \delta^{-1} \int_{u=2}^s (u-1)^{-M \delta} \rmd u \Big) \,,
\end{multlined}
\end{equation}
where we have used the explicit formula for the Laplace transform of $\cN$. If $M \delta > 1$, the integral up to $s=\infty$ will converge, making the above quantity uniformly bounded away from $0$ for all $s \ge 2$.

It remains to show that the expectation in~\eqref{e:215} is also uniformly positive. To this end, we recall that (or use the first part of Proposition~\ref{p:3.7}) that for all $z \leq -1$,
\begin{equation}
\bbP_{s,z} \big(\max_{u \in [s,2s]} W_u \leq -1) \leq Cs^{-1/2}(z^-+1) \,.
\end{equation}
Then using the total probability formula with respect to $W_s$ and using the above bound we get,
\begin{equation}
\begin{multlined}
\bbP_{0,-1} \big( \max_{u \in [0,2s]} \big(W_u + M \log^+ u\big) \leq 0\big) \\
\leq C s^{-1/2} \bbE_{0,-1} \Big(W_s^-+1 ;\; \max_{u \in [0,s]} \big(W_u + M \log^+ u \big) \leq 0\Big)
\end{multlined}
\end{equation}
By the second part of Proposition~\ref{p:3.7}, the probability on the left is at least $C s^{-1/2}$, while by Proposition~\ref{p:A2}, the expectation of $1$ on the event in the expectation on the right is at most $C' s^{-1/2}$. This shows that the expectation of $W_s^-$ on the same event is bounded away from $0$ for all $s$ large enough and by shift invariance the same is true for the expectation in~\eqref{e:215}. All together this shows the existence of $C > 0$ such that $f_{s+1}(x) >c$ for all $s$ large enough, which is the first statement in the lemma.

As for~\eqref{e:599}, again arguing only for $f$, if we define $\cQ_\infty(0,s; x)$ as in~\eqref{e:692} only with $W_{\sigma_k}+x$ in place of $W$, then clearly we can write $f_s(-x)$ as $\bbE_{0,0} \big((W_s - x)^- ;\; \cQ_\infty(0,s ;-x)\big)$. Dividing by $x$ and letting $x \to \infty$, the random variable inside the expectation converges almost surely to $1$ and bounded for all $x \geq 1$ by $|W_s|+1$. The result then follows from the bounded convergence theorem.
\end{proof}

Next, we prove Proposition~\ref{p:A4},
\begin{proof}[Proof of Proposition~\ref{p:A4}]
A quick look at the proof of Lemma~\ref{l:10.9}, shows that the outer most limit in~\eqref{e:196} is in fact uniform in all $Y, \gamma$ satisfying~\eqref{e:A1}~\eqref{e:A2.5},~\eqref{e:A3} and~\eqref{e:A4} for a fixed $\delta > 0$. Indeed, the overall rate of convergence in $s$ depends only on the rate of convergence in $s$ in~\eqref{e:A16}. In the rest of the argument, $s$ remains fixed. This in turn depends on the entropic repulsion bound given in Lemma~\ref{l:10.7}, which depends on $Y$ and $\gamma$ only through conditions~\eqref{e:A1} and~\eqref{e:A3} and as such naturally uniform for all $Y$, $\gamma$ satisfying these conditions with a given $\delta$. In particular, as in the proof of Proposition~\ref{p:A3}, this implies that $f_s(x)/(x^-+1) \longrightarrow f_\infty(x)$ and $g_s(y)/(y^-+1) \longrightarrow g_\infty(y)$ as $s \to \infty$, uniformly in $x,y < 1/\epsilon$ for fixed $\epsilon > 0$ and $Y, \gamma$ for fixed $\delta > 0$.

Now, denoting by $f_s^{(r)}$ and $g_s^{(r)}$ the versions of $f_s$ and $g_s$ respectively defined as in~\eqref{e:795} via the events in~\eqref{e:692} only with $Y^{(r)}$ and $\gamma^{(r)}$ in place of $Y$ and $\gamma$. Assumption~\eqref{e:543} in the proposition, the absolute continuity of the marginals of $W$ and the dominated convergence theorem ensure that for all $x,y \in \bbR$ and $s > 0$,
\begin{equation}
f^{(r)}_s(x) \overset{s \to \infty} \longrightarrow f_s(x) \quad , \qquad
g^{(r)}_s(x) \overset{s \to \infty} \longrightarrow g_s(x) \,,
\end{equation}
with $f_s$, $g_s$ defined with respect to $Y$ and $\gamma$ appearing in~\eqref{e:543}.

Thanks to the uniformity discussed above, we can exchange the order in which the limits are taken, thereby obtaining
\begin{equation}
\lim_{r \to \infty} f^{(r)}(x) = \lim_{r \to \infty} \lim_{s \to \infty} f^{(r)}_s(x) =
 \lim_{s \to \infty} \lim_{r \to \infty} f^{(r)}_s(x) = \lim_{s \to \infty} f_s(x) = f(x) \,, 
\end{equation}
with a similar statement for $g$. The last statement in the proposition is obvious in light of the definition of $g$.
\end{proof}

Finally, we prove Proposition~\ref{p:A5}.
\begin{proof}[Proof of Proposition~\ref{p:A5}]
The conditional probability is equal to the ratio between the probability in the statement of Lemma~\ref{l:10.7.0} and the probability in the statement of Proposition~\ref{p:A3}. In light of the upper bound in Lemma~\ref{l:10.7.0} and the limit in Proposition~\ref{p:A3}, noting that $f(x)$ and $g(y)$ are positive, the proof follows.
\end{proof}

\bibliographystyle{abbrv}
\bibliography{GBBStructure_Supp}

\end{document}